\documentclass[11pt,reqno]{amsart}
\usepackage{amsthm,amssymb,amsmath}
\usepackage{textcomp}
\usepackage[foot]{amsaddr}
\usepackage{xcolor}
\usepackage{hyperref}
\usepackage[T1]{fontenc}
\usepackage{lmodern}
\hypersetup{
	colorlinks =true , 
	linkcolor =blue,
	urlcolor = magenta,
	citecolor =blue
}

\newtheorem{theorem}{Theorem}[section]
\newtheorem*{theorem*}{Theorem}
\newtheorem{lemma}{Lemma}[section]

\theoremstyle{definition}

\newtheorem{definition}{\bf Definition}[section]
\newtheorem*{definition*}{\bf Definition}

\newtheorem{example}{\bf Example}[section]

\newtheorem{remark}{Remark}
\newtheorem*{remark*}{Remark}

\newtheorem*{example*}{\bf Example}

\begin{document}

	\title[]{On the Fountain Theorem for Continuous Functionals and Its Application to a Semilinear Elliptic Problem in $\mathbb{R}^2$}

	\author{Ablanvi Songo$^{1,a}$}

\address{$^a$Universit\'{e} de Sherbrooke, D\'{e}partement de math\'{e}matiques, Sherbrooke, Qu\'{e}bec, Canada}

\email{\href{mailto:ablanvi.songo@usherbrooke.ca}{{\textcolor{blue}{ablanvi.songo@usherbrooke.ca}}}}
\footnote{Corresponding author}
\author{ Fabrice Colin${^b}$}

\address{$^b$Laurentian University, School of Engineering and Computer Science, Sudbury, Ontario, Canada}

\email{\href{mailto:fcolin@laurentian.ca}{{\textcolor{blue}{fcolin@laurentian.ca}}}}



	\begin{abstract}
		In this work, we establish a continuous version of the Fountain theorem by using the framework of the weak slope for continuous functionals, which generalizes Theorem 3.6 of Willem \cite{Wi}. Then we present an application to a semilinear problem. 
	\end{abstract}
	
	\maketitle

	\section{Introduction}
	Many papers are devoted to the study of critical point theory for nonsmooth functionals, see, for example, \cite{CD, CDM, CD}, \cite{Gu}, \cite{Ch}, \cite{KV}. In this paper, we use the deformation lemma for continuous functions which are invariant under a topological group found in \cite{KV} to prove the Willem minimax principle theorem \cite[Theorem 3.5]{Wi} for continuous functionals. With the aid of this result, we give a nonsmooth version of the Fountain theorem (see Theorem 3.2) which generalizes the Fountain theorem of Willem \cite[Theorem 3.6]{Wi}.
    
    As application, we study the existence of multiple solutions of the following semilinear equation
		\begin{equation*}
		(P)
		\begin{cases}
			-\Delta u + u = f(x,u),\\
			u\in H^1(\mathbb{R}^2),
		\end{cases}
	\end{equation*} 
	where the nonlinearity $f \in \mathcal{C}(\mathbb{R}^2 \times \mathbb{R}, \mathbb{R})$ satisfies some growth conditions, $\Delta$ is the Laplace operator, and the space $H^1(\mathbb{R}^2)$ is a Hilbert space. The main difficulty here lies in the lack of compactness due to the unboundedness of the domain, as well as the fact that the nonlinear term $f(x,u)$ exhibits subcritical exponential growth in the sense of the Trudinger–Moser inequality. By applying our main result (see Theorem 3.2) together with the Trudinger–Moser inequality \cite{OS}, we prove that problem $(P)$ admits multiple solutions.\\
    
	The paper is organized as follows. In Section 2, we present the concept of the weak slope for continuous functions. In Section 3, we give the main result of this paper (see Theorem 3.2). In Section 4, we give an application of our main result, that is, a continuous version of the Fountain theorem. 
	
	\section{Preliminaries}
 In this section we give some preliminary information and results which will be used later.\\ Let $(E,d)$ be a metric space endowed with the metric $d$. We recall some definitions and results from the paper \cite{DM}, see also \cite{CD, CDM}.
	\begin{definition}
		Let $f: E \to \mathbb{R}$ be a continuous function and let $u\in E$. We denote by $|df|(u)$ the supremum of $\sigma \in [0, \infty)$ such that there exist $\delta >0$ and a continuous map 
		\begin{equation*}
			\mathcal{H} : B(u,\delta)\times [0,\delta] \rightarrow E
		\end{equation*}
		such that  for all  $v \in B(u,\delta)$, for all $ t\in [0,\delta]$ we have 
		\begin{enumerate}
	\item $d(\mathcal{H}(v,t),v) \le t$,
	\item  $f(\mathcal{H}(v,t))\le f(v) -\sigma t$,
		\end{enumerate}
		where \[B(u,r) := \Big\{w\in E\;|\; d(u,w)<r\Big\}\;\;\text{with} \;\;  r>0.\] The extended real number $|df|(u)$ is called the weak slope of $f$ at $u$.
	\end{definition}
	\begin{definition}
		Let $f: E \to \mathbb{R}$ be a continuous function.
	 We say that $u\in E$ is a critical point for $f$ if $|df|(u) = 0$. A real number $c$ is said to be a critical value for $f$, if there exists $u\in E$ such that $|df|(u)=0$ and $f(u)=c$.
	
	\end{definition}
	
	\begin{definition}[see \cite{DM}]
		Let $f: E \to \mathbb{R}$ be a continuous function and $c \in \mathbb{R}$. We say that $f$ satisfies the Palais-Smale condition at level $c$ ($(PS)_c$ for short), if from every sequence $(u_h)$ in $E$ with $|df|(u_h)\to 0$ and $f(u_h)\to c$ as $h\to \infty$ it is possible to extract a subsequence $(u_{h_j})$ converging in $E$ (the limit of $(u_{h_j})$ is necessarily a critical point for $f$, see \cite[Proposition 2.6]{DM} ).
	\end{definition}
	Let $G$ be a group of isometries of $E$, that is 
	\begin{equation*}
	G:= \Big\{g: E\to E\;|\; d(g(u),g(v))=d(u,v), \;\;\text{for all}\;\; u, v\in E \Big \}.
	\end{equation*}

	We say that 
	\begin{equation*}
		A\subset E\;\; \text{is}\;\; G-\text{invariant if}\;\; g(A)=A \;\; \text{for all}\;\; g \in G,
	\end{equation*}
	\begin{equation*}
		h: E\to \mathbb{R}\;\; \text{is}\;\; G-\text{invariant if}\;\; h\circ g= h \;\; \text{for all}\;\; g \in G,
	\end{equation*}
	\begin{equation*}
		h: E\to E \;\; \text{is}\;\; G-\text{equivariant if}\;\; h\circ g= g\circ h \;\; \text{for all}\;\; g \in G.
	\end{equation*}
	\begin{definition}[Weak slope for the equivariant case]
	Let $f: E \to \mathbb{R}$ be a continuous $G-$invariant function. For every $u\in E$ we denote by $|d_Gf|(u)$ the supremum of $\sigma \in [0, \infty)$ such that there exist an $G-$invariant neighborhood $U$ of $u$, as well as $\delta >0$ and a continuous map 
		\begin{equation*}
			\mathcal{H} : U\times [0,\delta] \rightarrow E
		\end{equation*}
		such that  for all  $v \in U$, for all $ t\in [0,\delta]$ we have 
		\begin{enumerate}
			\item  $\mathcal{H}(.,t)$ is $G-$equivariant for all $t\in [0,\delta]$,
			\item $d(\mathcal{H}(v,t),v) \le t$,
			\item  $f(\mathcal{H}(v,t))\le f(v) -\sigma t$.
		\end{enumerate}
		The extended real number $|d_Gf|(u)$ is called the equivariant weak slope of $f$ at $u$.
	\end{definition}
    \begin{remark}
        We have \[|d_Gf|(u)\le |df|(u).\]
    \end{remark}
	\begin{definition}
			Let $f: E \to \mathbb{R}$ be a continuous $G-$invariant function.
		We say that $u\in E$ is a $G-$critical point for $f$ if $|d_Gf|(u) = 0$. A real number $c$ is said to be a $G-$critical value for $f$, if there exists $u\in E$ such that $|d_Gf|(u)=0$ and $f(u)=c$.
		
	\end{definition}
	\begin{definition}[see \cite{CD}]
    \label{def 2.6}
	Let	$f: E \to \mathbb{R}$ be a continuous $G-$invariant function. Let $c\in \mathbb{R}$. We say that $f$ satisfies the $G-$Palais-Smale condition ($G-(PS)_c$ for short) if from every sequence $(u_h)$ in $E$ with $|d_Gf|(u_h)\to 0$ and $f(u_h)\to c$ it is possible to extract a subsequence $(u_{h_j})$ converging in $E$ (the limit of $(u_{h_j})$ is necessarily a $G-$critical point for $f$, see \cite[Theorem 1.2.2]{CD}).
	\end{definition}
	
The following theorem is from  \cite[Theorem 1.1.2]{CD} and is crucial to prove that the energy functional associated with the problem fulfills the $G-(PS)_c$ condition. 
\begin{theorem}
	\label{theo 2.1}
	Let $X$ be an open subset of a normed space $E$ and let $u\in X$, $v\in X$, and $f: E \to \mathbb{R}$ be a continuous function. For any $w\in X$, set
	\begin{equation*}
	\overline{D}_{+} f(w)(v):= \underset{t\to 0^+}{\limsup}\; \dfrac{f(w+tv)-f(w)}{t}.
	\end{equation*}
	Then we have 
	\begin{equation*}
		\|v\||df|(u)\ge -\underset{w\to u}{\limsup}\;\overline{D}_{+} f(w)(v).\\
	\end{equation*}
\end{theorem}
	\begin{remark}
    \label{rmk 2}
	If the function $f$ is $G-$invariant with $G= \mathbb{Z}_2 \simeq \{id, -id\}$, then \[|d_{\mathbb{Z}_2}f|(u)=|df|(u).\] Therefore, Theorem $\ref{theo 2.1}$ is also valid in this case.
	\end{remark}

   Let $S \subset E$ and $f : E \rightarrow \mathbb{R}$ be a continuous $G-$invariant functional, where $(E,d)$ is a metric space on which a compact Lie group $G$ acts by isometric transformations (a $G-$metric space for short).

    We denote by 
   \begin{equation*}
        d(u, S):= \|u-S\|.
   \end{equation*}
   For every $\alpha$, $\beta \in \mathbb{R}$, denote
   \begin{equation*}
       S_{\alpha} := \{u \in E \;|\; d(u,S) \le \alpha\},\;\;
       f_{\beta} := \{u \in E \;|\; f(u)\ge \beta \},
\end{equation*}
       \begin{equation*}
       f^{\alpha} :=\{u \in E \;|\; f(u)\le \alpha \}, \;\;
       f_\beta^{\alpha}:= f_\beta \cap f^\alpha.
   \end{equation*}

We recall the following equivariant deformation lemma from \cite[Thorem 4.4]{KV} for continuous and $G-$invariant functions which will allow us to generalize the general minimax principle for continuous functionals.

\begin{theorem}[Theorem 4.4, \cite{KV}]
	\label{theo 2.2}
	Let $(E,d)$ be a complete $G-$metric space, $f:E\to \mathbb{R}$ be a continuous $G-$invariant function, $C$ a closed $G-$invariant subset in $E$ and $c\in \mathbb{R}$. Let $\varepsilon>0$ be such that 
	\begin{equation*}
		\text{for every}\;\; u\in f^{-1}\Big([c-2\varepsilon, c+2\varepsilon]\Big)\cap C_{2\varepsilon}, \;\; \text{we have}\;\; |d_Gf|(u)>\varepsilon.
	\end{equation*}
	Then, there exist $\varepsilon' \in (0,\varepsilon)$, as well as $\lambda>0$ and a continuous $G-$equivariant map 
    \begin{equation*}
        \eta: [0,1]\times E \to E
    \end{equation*} such that\\
	$(a)$ $d(\eta(t,u),u)\le \lambda t$, for every $t\in [0,1]$;\\
	$(b)$ $f(\eta(t,u))\le f(u)$, for every $t\in [0,1]$ and $u\in E$;\\
	$(c)$ if $u\notin f^{-1}\Big([c-2\varepsilon, c+2\varepsilon]\Big) \cap C_{2\varepsilon}$, then $\eta(t,u)=u$, for every $t\in [0,1]$;\\
	$(d)$ $\eta(1, f^{c+\varepsilon'}\cap C)\subset f^{c-\varepsilon'}$ with $\varepsilon' =\frac{\varepsilon}{2\sqrt{1+\varepsilon^2}}$; \\
	$(e)$ for every $t\in]0,1]$ and for every $u\in f^c \cap C$ we have $f(\eta(t,u))<c$.
\end{theorem}

	\section{Main results}
	
	\begin{definition}
		Assume that a compact group $G$ acts diagonally on $V^k$, that is,
		\begin{equation*}
			g(v_1, \cdots, v_k):= (gv_1, \cdots, gv_k)
		\end{equation*} 
		where $V$ is a finite-dimensional space. Let $k\ge 2$. The action of $G$ is admissible if every continuous $G-$equivariant map $\partial U \to V^{k-1}$ has a zero. Here, $U$ may be any open bounded $G-$invariant neighborhood of $0$ in $V^k$.
	\end{definition}
		\begin{example}
			The Borsuk-Ulam antipodal theorem says that the antipodal action of $G:= \mathbb{Z}_2$ on $V:=\mathbb{R}$ is admissible.
		\end{example}

		We consider the following situation: \\
		
		$(A_1)$ The compact group $G$ acts isometrically on the Banach space $E= \overline{\underset{j\in \mathbb{N}}{\oplus} E_j}$, the spaces $E_j$ are $G-$invariant and there exists a finite-dimensional space $V$ such that, for every $j\in \mathbb{N}, E_j \simeq V$ and the action of $G$ on $V$ is admissible.\\
		
		In the rest of the paper, we use the following notation:
		\begin{equation*}
			Y_k := \oplus_{j=0}^{k}E_j, \;\; Z_k:= \overline{\oplus_{j=k}^{\infty} E_j},
		\end{equation*}
	\begin{equation*}
		B_k:= \Big\{u\in Y_k\;|\; \|u\|\le \rho_k\Big\}, \;\; N_k:=\Big\{u\in Z_k\;|\; \|u\| = r_k\Big\},
	\end{equation*}
	where $\rho_k> r_k>0$ are fixed.
	\begin{lemma}[Intersection lemma, see \cite{Wi}]
		\label{lem 3.1}
		Under assumption $(A_1)$, if $\gamma\in \mathcal{C}(B_k, E)$ is $G-$equivariant and if $\gamma|_{\partial B_k}=\operatorname{id}$, then \[\gamma(B_k)\cap N_k \ne \emptyset.\]
	\end{lemma}
    \subsection{A general minimax result}
    The next theorem is a generalization of \cite[Theorem 3.5]{Wi}.
	\begin{theorem}[General minimax principle]
		\label{theo 3.1}
Under assumption $(A_1)$, let $f : E\to \mathbb{R}$ be a continuous and $G-$invariant functional. Define for every $k\ge 2$, 
\begin{equation}
\label{eq 1}
	c_k:= \underset{\gamma\in \Gamma_k}{\inf}\underset{u\in B_k}{\max}\; f(\gamma(u)),
\end{equation}
where \begin{equation*}
	\Gamma_k:= \Big\{\gamma\in \mathcal{C}(B_k, E)\;\Big|\; \gamma \;\; \text{is $G-$equivariant and}\;\; \gamma \Big|_{\partial B_k}=\operatorname{id}\Big \}.
\end{equation*}
If \begin{equation*}
	b_k:= \underset{\underset{\|u\|=r_k}{u\in Z_k}}{\inf} \; f(u)> a_k:= \underset{\underset{\|u\|=\rho_k}{u\in Y_k}}{\max}\; f(u),
\end{equation*}
then $c_k \ge b_k$ and, for every $\varepsilon \in (0,\frac{c_k-a_k}{2})$ and for every  $\gamma \in \Gamma_k$ such that 
\begin{equation}
	\label{eq 2}
	\underset{u\in B_k}{\max}\; f(\gamma(u))\le c_k + \varepsilon', \;\; \text{for some}\;\; \varepsilon' \in (0,\varepsilon),
\end{equation}
 there exists $u\in E$ such that 
 \begin{enumerate}
 	\item[(a)] \[c_k-2\varepsilon \le f(u)\le c_k+2\varepsilon,\]
 	\item[(b)] \[d(u, \gamma(B_k))\le 2\varepsilon,\]
 	\item[(c)] \[|d_Gf|(u)\le \varepsilon.\] 
 	
 \end{enumerate}
	\end{theorem}
	\begin{proof}
		By the intersection lemma (Lemma $\ref{lem 3.1}$), $c_k\ge b_k$. Suppose that the conclusion of the theorem is false. Then, there are $\varepsilon \in (0, \frac{c_k-a_k}{2})$ and $\gamma \in \Gamma_k$ satisfying
\begin{equation}
\label{eq 3}
   \underset{u \in B_k}{\max}\; f(\gamma(u)) \le c_k +\varepsilon' 
\end{equation}
	such that for every $u \in f^{-1} \Big([c_k-2\varepsilon,  c_k+2\varepsilon]\Big) \cap \Big(\gamma(B_k)\Big)_{2\varepsilon}$, we have \[|d_G f|(u) > \varepsilon.\]
	Then, Theorem $\ref{theo 2.2}$ can be applied with $C := \gamma(B_k)$. Since $\varepsilon \in (0, \frac{c_k-a_k}{2})$, then
		\begin{equation} 
			\label{eq 4}
			c_k-2\varepsilon > a_k.
		\end{equation}
		We consider the map
        \begin{eqnarray*}
        \beta &:& B_k\to E\\
            u &\mapsto& \eta(1, \gamma(u)),
        \end{eqnarray*}
        where $\eta$ is given by Theorem $\ref{theo 2.2}$. \\ For every $u\in \partial B_k$, since $\gamma(u)=u$, we obtain from $(\ref{eq 4})$ that \[f(\gamma(u)) \le a_k<c_k-2\varepsilon.\] Therefore, \[\gamma(u)\notin f^{-1}\Big([c_k-2\varepsilon, c_k+2\varepsilon]\Big).\] By $(c)$ of Theorem $\ref{theo 2.2}$ we have 
		\begin{equation*}
			\beta(u)=\eta(1,\gamma(u))=\gamma(u)=u.
		\end{equation*}
		Since $\eta$ is $G-$equivariant, it follows that $\beta$ is equivariant, so $\beta \in \Gamma_k$. 
        Since for every $u\in B_k$, $(\ref{eq 3})$ implies $\gamma (u) \in f^{c_k+\varepsilon'}$, that is, $\gamma(u)\in C\cap f^{c_k+\varepsilon'}$, then by $(d)$ of Theorem $\ref{theo 2.2}$, \[\eta(1, \gamma(u))\in f^{c_k-\varepsilon'}.\] This implies that
		\begin{equation*}
			\underset{u\in B_k}{\max}\; f(\beta(u)) = \underset{u\in B_k}{\max}\; f\Big(\eta(1,\gamma(u))\Big) \le c_k - \varepsilon'.
		\end{equation*}
	Finally, by the definition of $c_k$, we have
    \begin{equation*}
       c_k \le \underset{u\in B_k}{\max}\; f(\beta(u)) = \underset{u\in B_k}{\max}\; f\Big(\eta(1,\gamma(u))\Big) \le c_k - \varepsilon'.
    \end{equation*}
    This is a contradiction.
	\end{proof}
    \subsection{Fountain theorem for continuous functionals}
    The following theorem is a generalization of the Fountain theorem.
	\begin{theorem}[Fountain theorem for continuous functionals]
		\label{theo 3.2}
		Under assumption $(A_1)$, let $f: E\to \mathbb{R}$ be a continuous and $G-$invariant functional. If, for every $k\in \mathbb{N}$, there exists $\rho_k >r_k>0$ such that 
		\begin{enumerate}
			\item[$(A_2)$] $a_k:= \underset{u\in \partial B_k}{\max}\; f(u)\le 0$,
			\item[$(A_3)$] $b_k:= \underset{u\in N_k}{\inf} \; f(u) \to \infty,\; k\to \infty$,
			\item[$(A_4)$] the functional $f$ satisfies $G-(PS)_c$ condition for every $c>0$,
		\end{enumerate}
        then $f$ has an unbounded sequence of $G-$critical values.
	\end{theorem}
	\begin{proof}
		For $k$ large enough, $b_k>0$. Theorem $\ref{theo 3.1}$ implies then the existence of a sequence $(u_n)\subset E$ satisfying 
		\begin{equation*}
			f(u_n)\to c_k,\;\; |d_G f|(u_n)\to 0,
		\end{equation*}
        where $c_k$ is defined in $(\ref{eq 1})$.
		It follows from $(A_4)$ that $c_k$ is a $G-$critical value for $f$. Since $c_k\ge b_k$, by $(A_3)$, $c_k \to \infty$. The proof of the theorem is complete.
	\end{proof}
	\section{Application of the Fountain theorem for continuous functionals}
	In this section, the continuous version of the Fountain theorem is applied to the problem
	
	\begin{equation*}
		(P)
		\begin{cases}
		-\Delta u + u = f(x,u),\\
		u\in H^1(\mathbb{R}^2),
			\end{cases}
	\end{equation*} 
	where the nonlineear function $f \in \mathcal{C}(\mathbb{R}^2 \times \mathbb{R}, \mathbb{R})$ satisfies some growth conditions, $\Delta$ is the Laplace operator and the space $H^1(\mathbb{R}^2)$ is the Hilbert space defined by
	\begin{equation*}
		H^1(\mathbb{R}^2) := \Big\{u\in L^2(\mathbb{R}^2)\;\Big|\; |\nabla u| \in L^2(\mathbb{R}^2) \Big\},
	\end{equation*}
	endowed with the inner product 
	\begin{equation*}
		(u,v) = \int_{\mathbb{R}^2} \Big(\nabla u \nabla v + uv\Big) dx,
	\end{equation*}
	and the associated norm $\|.\|$ given by
	\begin{equation*}
	\|u\|^2 :=	\int_{\mathbb{R}^2} \Big(|\nabla u |^2 + |u|^2\Big) dx.
	\end{equation*}
We define the functional 
	\begin{equation}
    \label{eq 5}
		J(u) = \dfrac{1}{2} \int_{\mathbb{R}^2} \Big(|\nabla u(x)|^2 + u(x)^2 \Big)dx - \int_{\mathbb{R}^2} F(x,u(x))dx
	\end{equation}
	with 
	\begin{equation*}
		F(x,u(x)) := \int_{0}^{u(x)}f(x,t)dt.
	\end{equation*}
In this section, $|\cdot|_p$ stands for the $L^p$-norm; $\to$ and $\rightharpoonup$ denote strong and weak convergence, respectively.
    
	\begin{definition}
    \label{def 4.1}
			We say that $u$ is a weak solution of problem $(P)$ if $u \in H^1(\mathbb{R}^2)$, and satisfies for any $ w \in H^1(\mathbb{R}^2)$ : 
		\begin{equation*}
			\int_{\mathbb{R}^2} \Big(\nabla u \nabla w +  u w \Big)dx - \int_{\mathbb{R}^2} f(x,u)w dx =0.\\
		\end{equation*}
        \end{definition}

In order to overcome the loss of compactness of the Sobolev embedding involving $H^1(\mathbb{R}^2)$, we are going to work with the space $H^1_{rad}(\mathbb{R}^2)$ defined by 
\begin{equation*}
    H^1_{rad}(\mathbb{R}^2):= \Big\{u\in H^1(\mathbb{R}^2) \; |\; u(x) =u(|x|),\;\; \text{alomost everywhere, \;in }\;\; \mathbb{R}^2 \Big\}.
\end{equation*}

It is well known (see \cite{Wi}) that $H^1_{rad}(\mathbb{R}^2)$ is a closed vector subspace of $H^1(\mathbb{R}^2)$, and so, $H^1_{rad}(\mathbb{R}^2)$ is a Hilbert space with the norm of $H^1(\mathbb{R}^2)$. The space $H^1_{rad}(\mathbb{R}^2)$ is called the subspace of radial functions of $H^1(\mathbb{R}^2)$. Moreover, there holds the following compact embedding due to Strauss \cite{S}.
\begin{theorem}[Strauss, 1977]
\label{theo 4.1}
Let $2<p<\infty$. The embeddings
\begin{equation*}
    H^1_{rad}(\mathbb{R}^2) \subset L^p(\mathbb{R}^2)
\end{equation*}
are compact.
\end{theorem}
\begin{definition}
    Let $O(2)$ be the group of all linear orthogonal transformations of $\mathbb{R}^2$. Let $G_0\subset O(2)$ be acting on $H^1(\mathbb{R}^2)$ by $g_0u(x):=u(g_0^{-1}x)$, for all $g_0 \in G_0$. We define the space of $G_0-$invariant points of $H^1(\mathbb{R}^2)$ by 
    \begin{equation*}
        Fix(G_0):=\Big\{u\in H^1(\mathbb{R}^2)\;|\; g_0u=u, \;\; \text{for every}\;\; g_0\in G_0 \Big\},
    \end{equation*}
    where $O(2)$ is endowed with the composition operation.
\end{definition}

\begin{remark}
    If $G_0=O(2)$, then 
    \begin{equation*}
    Fix(G_0) =H^1_{rad}(\mathbb{R}^2). 
    \end{equation*}
\end{remark}

$ (\star)$ From now on, we set $E:= H^1_{rad}(\mathbb{R}^2)$. We choose an orthonormal basis $(e_j)$ of $E$ and define $E_j:= \mathbb{R}e_j$. In $H^1(\mathbb{R}^2)$, we consider the natural antipodal action of the group \[G = \mathbb{Z}_2 \simeq \{ \mathrm{id}, -\mathrm{id} \},\] defined by \[( -\mathrm{id}) \cdot u = -u.\]\\

Before we state the main result of this section, we recall the following principle of symmetric criticality theorem from Palais (1979).
\begin{theorem}[Theorem 1.28, \cite{Wi}]
\label{theo 4.2}
    Assume that the action of the topological group $G_0$ on the Hilbert space $X$ is isometric. If $\varphi\in \mathcal{C}^1(X,\mathbb{R})$ is $G_0-$invariant and $u$ is a $G-$critical point of $\varphi$ restricted to $Fix(G_0)$, then $u$ is a $G-$critical point of $\varphi$ on $X$.
\end{theorem}
\subsection{An existence result}
	Here is the main result of this section :
    
	\begin{theorem}
		\label{theo 4.3}
For the function $f$ from $(P)$, asumme that $f  \in \mathcal{C}(\mathbb{R}^2 \times \mathbb{R}, \mathbb{R})$.\\	Assume also that 
	\begin{enumerate} 
	\item[$(f_0)$]  $f$ has subcritical growth at $ t\to +\infty$, that is, if for all $\beta>0$ we have 
    \begin{equation*}
        \lim_{t\to +\infty} \dfrac{f(x,t)}{e^{\beta t^2}} =0,
    \end{equation*}
\item[$(f_1)$] 

    $\lim_{t\to 0} \dfrac{f(x,t)}{t} =0, \;\; \text{uniformly with respect to}\;\; x\in \mathbb{R}^2$,
	\item[$(f_2)$] there are $\mu>2$ and $R>0$ such that, for every $x\in \mathbb{R}^2$ and for every $t\in \mathbb{R}$ with $|t|\ge R$ we have
	\begin{equation*}
		0<\mu F(x,t)\le tf(x,t),
	\end{equation*}
    where \begin{equation*}
        F(x,t)=\int_0^{t}f(x,s)ds,
    \end{equation*}
    \item[$(f_3)$] $f(g_0x, t)= f(x,t)$, for every $x\in \mathbb{R}^2$, for every $g_0\in O(2)$ and for every $t\in \mathbb{R}$,
	\item[$(f_4)$] $f(x,-t)=-f(x,t)$, for every $x\in \mathbb{R}^2$, for every $t\in \mathbb{R}$.\\
	\end{enumerate}	
	
	Then problem $(P)$ has a sequence of weak solutions $(u_k)$ in $H^1(\mathbb{R}^2)$ such that $J(u_k)\to \infty$, as $k\to \infty$, where $J$ is given in $(\ref{eq 5})$.\\
	\end{theorem}
    \begin{remark}
       We would like to mention that condition $(f_0)$ in Theorem $\ref{theo 4.3}$ was considered by several authors; see for example \cite{ZL}, \cite{FDR} and \cite{dMS}. In \cite{ZL}, the condition $(f_0)$ is replaced by the following condition: for every $\beta>0$, there exists $c>0$ such that $f(x,t)\le c e^{\beta t^2}$ for every $(x,t) \in \mathbb{R}^2 \times [0,\infty[$.\\
      \begin{example}
       Let $R>0$, let $x\in \mathbb{R}^2$. Let $\sigma>2$. Define  $\zeta: \mathbb{R}^2 \to (0,\infty), \;\;\zeta(x)=1+e^{-|x|^2}$.
        Typical examples of functions satisfying the assumptions $(f_0)$, $(f_1)$, $(f_2)$, $(f_3)$ and $(f_4)$ are
       \begin{equation*}
       f(x,t)=\zeta(x)
           \begin{cases}
               |t|^{\sigma-2}t, \;\;\text{if}\;\; |t|\ge R\\
               t^3, \;\;\text{if}\;\;|t|<R
           \end{cases}
       \end{equation*}
       and 
        \begin{equation*}
       f(x,t)=\zeta(x)
           \begin{cases}
               t(2+|t|)e^{|t|}, \;\;\text{if}\;\; |t|\ge R\\
               t^3e^{|t|}, \;\; \text{if}\;\; |t|<R
           \end{cases}
       \end{equation*}
      \end{example}
    \end{remark}
	To prepare for the proof of Theorem $\ref{theo 4.3}$, we first establish several lemmas. \\
    
	In 2001, do Ó and Souto \cite{OS} proved in the whole space $\mathbb{R}^2$ a version of the Trudinger-Moser inequality, namely, 
	\begin{enumerate}
		\item  if  $u\in H^1(\mathbb{R}^2)$ and  $\beta>0$, then
		\begin{equation}
			\label{eq 6}
			\int_{\mathbb{R}^2}\big(e^{\beta |u|^2}-1\big)dx < + \infty;
		\end{equation}
		\item if $0< \beta < 4\pi$ and $|u|_{L^2(\mathbb{R}^2)} \le m_0 $, then there exists a constant $n_0$, which depends only on $\beta$ and $m_0$, such that 
		\begin{equation*}
			\underset{|\nabla u|_{L^2(\mathbb{R}^2)} \le 1}{\sup} \int_{\mathbb{R}^2}\big(e^{\beta |u|^2}-1\big)dx \le n_0.\\
		\end{equation*}
		\end{enumerate}
        
		We will need the following standard result.
		\begin{lemma}
			\label{lem 4.1}
			Let $g\in L^1_{loc}(\mathbb{R})$. For $y_0\in \mathbb{R}$ fixed, set 
			\begin{equation*}
				v(t)=\int_{y_0}^{t}g(s)ds, \;\; t\in \mathbb{R}.
			\end{equation*}
			Then $v$ is a continuous function on $\mathbb{R}$.
		\end{lemma}
	\begin{lemma}
		\label{lem 4.2}
		Let $f$ satisfies assumptions $(f_0)$ and $(f_1)$. Then, for given $\varepsilon >0$, there exists $ c = c(\varepsilon)  >0$ such that 
		\begin{equation*}
			|f(x,t)|\le 2\varepsilon |t|+ c\big(e^{\beta t^2}-1\big), \;\; \text{for every}\;\; (x,t)\in \mathbb{R}^2\times \mathbb{R}.
            \end{equation*}
	\end{lemma}
	\begin{proof}
		From $(f_1)$, for given $\varepsilon>0$, there is $c_\varepsilon>0$ such that
        \begin{equation}
        \label{eq 7}
            |f(x,t)|\le 2\varepsilon |t|, \;\; \text{for every}\;\; |t|\le c_\varepsilon,
        \end{equation}
        and from $(f_0)$, for each $\beta>0$ we have 
        \begin{equation}
        \label{eq 8}
            |f(x,t)| \le \varepsilon e^{\beta t^2}, \;\; \text{for every}\;\; |t|> c_\varepsilon.
        \end{equation}
        Now, from $(\ref{eq 8})$, there exists $c = c(\varepsilon)>0$ such that 
        \begin{equation}
        \label{eq 9}
            |f(x,t)| \le c \big(e^{\beta t^2}-1\big), \;\; \text{for every}\;\; |t|> c_\varepsilon.
        \end{equation}
        The result follows from  $(\ref{eq 9})$ and $(\ref{eq 7})$.
        \end{proof}
        
        \begin{lemma} 
        \label{lem 4.3}
        Let $f$ satisfies assumptions $(f_0)$ and $(f_1)$. Then, for given $\varepsilon>0$, for each $\delta\ge1$, there is $C=C(\varepsilon, \delta)>0$ such that 
        \begin{equation*}
            |F(x,t)|\le \varepsilon t^2 + C|t|^\delta \Big( e^{\beta t^2}-1\Big), \;\; \text{for every}\;\; (x,t)\in \mathbb{R}^2\times \mathbb{R},
        \end{equation*}
    where
	\begin{equation*}
		F(x,t) := \int_{0}^{t}f(x,s)ds.
	\end{equation*}
        \end{lemma}
        \begin{proof}
        From $(\ref{eq 7})$, for given $\varepsilon>0$, there is $c_\varepsilon>0$ such that 
        \begin{equation*}
            |F(x,t)|\le \varepsilon t^2, \;\; \text{for every}\;\; |t|\le c_\varepsilon,
        \end{equation*}
        and from $(\ref{eq 9})$, for each $\delta\ge 1$, there is $C=C(\varepsilon, \delta)>0$ such that
        \begin{equation*}
            |F(x,t)|\le C|t|^\delta \Big( e^{\beta t^2}-1\Big), \;\; \text{for every}\;\; |t|> c_\varepsilon.
        \end{equation*}
        The lemma is proved.
        \end{proof}
        The next result is from \cite{dMS}.
        \begin{lemma}[\cite{dMS}, Lemma 2.2]
        \label{lem 4.4}
            Let $\beta>0$ and $r>1$. Then for each $\alpha>0$, there exists a positive constant $C=C(\alpha)$ such that for all $t\in \mathbb{R}$ we have 
            \begin{equation*}
              \Big( e^{\beta t^2}-1\Big)^r \le C \Big( e^{\alpha\beta t^2}-1\Big).  
            \end{equation*}
 In particular, if $u\in H^1(\mathbb{R}^2)$ then $\Big( e^{\beta u^2}-1\Big)^r$ belongs to $L^1(\mathbb{R}^2)$.
        \end{lemma}
        \begin{remark}
        \label{rm 5}
            As a consequence of $(\ref{eq 6})$ and Lemma $\ref{lem 4.4}$ and Hölder inequality, if $\beta>0$ and $q>0$, then the function $|u|^q \Big( e^{\beta u^2}-1\Big)$ belongs to $L^1(\mathbb{R}^2)$ for all $u\in H^1(\mathbb{R}^2)$.\\
            Indeed, let $u\in H^1(\mathbb{R}^2)$ and let $q>0$. Choose $q_0>1$ such that $\alpha q\ge 2$ where $\dfrac{1}{q_0}= 1- \dfrac{1}{\alpha}$. By Hölder inequality and for all $\alpha_0> q_0$, we have 
           \begin{align*}
               \int_{\mathbb{R}^2}|u|^q \Big( e^{\beta u^2}-1\Big) dt &\le \Bigg[   \int_{\mathbb{R}^2} \Big( e^{\beta u^2}-1\Big)^{q_0}dt\Bigg]^{\frac{1}{q_0}}\Bigg[ \int_{\mathbb{R}^2} |u|^{\alpha q} dt\Bigg]^{\frac{1}{\alpha}}\\
               &\le C (\alpha_0) \Bigg[  \int_{\mathbb{R}^2} \Big( e^{\alpha_0 \beta u^2}-1\Big)dt\Bigg]^{\frac{1}{q_0}} \Bigg[ \Big(\int_{\mathbb{R}^2} |u|^{\alpha q} dt\Big)^{\frac{1}{\alpha q}}\Bigg]^{q}\\
               &\le C_0 |u|_{\alpha q}^{q}.
               \end{align*}
          for some $C_0>0$. By the continuous Sobolev embedding theorem, there exists $c_1>0$ such that $|u|_{\alpha q}^{q} \le c_1 \|u\|^q$.\\
          We obtain
          \begin{equation*}
              \int_{\mathbb{R}^2}|u|^q \Big( e^{\beta u^2}-1\Big) dt \le c_2 \|u\|^q < \infty,
          \end{equation*}
          for some $c_2>0$.
        \end{remark}
	\begin{lemma}
    \label{lem 4.5}
		Let $f$ satisfies assumptions $(f_0),(f_1),(f_3)$ and $(f_4)$. Then, the functional $J(u)$ given in $(\ref{eq 5})$ is well defined, even, continuous, and $O(2)-$invariant on $E$.
	\end{lemma}
	\begin{proof}
		(1) \textbf{The functional $J$ is well defined} : \\
By Lemma $\ref{lem 4.3}$ and Remark $\ref{rm 5}$ and $(\ref{eq 6})$, for every $u\in E_{rad}$, we have
\begin{equation*}
	\int_{\mathbb{R}^2}|F(x,u)| dx < \infty. 
\end{equation*}
Therefore 
\begin{equation*}
	J(u)= \dfrac{\|u\|^2}{2} - \int_{\mathbb{R}^2} F(x,u(x))\, dx \in \mathbb{R}.
\end{equation*}
(2) \textbf{The functional $J(u)$ is even }:

	By assumption $(f_4)$, we have
	\begin{equation*}
		F(x,-u):= \int_{0}^{-u(x)}f(x,t)\,dt
		= \int_{0}^{u(x)}f(x,-s)\,d(-s)
		=- \int_{0}^{u(x)}(-f(x,s))\,ds
		= F(x,u(x)).
	\end{equation*}
	So, for every $u\in E$,
	\begin{equation*}
	J(-u)= \frac{\|-u\|^2}{2}- \int_{\mathbb{R}^2}F(x,-u)\,dx = J(u).
	\end{equation*}
(3) \textbf{The functional $J$ is continuous} : according to the assumption $(f_0)$, $f(x,.)\in L^1_{loc}(\mathbb{R})$ for every $x\in \mathbb{R}^2$. By Lemma $\ref{lem 4.1}$, for every $x \in \mathbb{R}^2$, $F(x,.)$ is continuous on $\mathbb{R}$.\\
Let $(u_n)\subset H^1(\mathbb{R}^2)$. Assume that $u_n\to u$ in $H^1(\mathbb{R}^2)$. Then, by the Sobolev embedding theorem, \[u_n\to u \;\; \text{in}\;\; L^p(\mathbb{R}^2), \;\;\text{for}\;\; 2\le p<\infty.\] For a subsequence still denoted by $(u_n)$, we have 
\begin{equation*}
	(u_n(x)) \to u(x), \;\; \text{almost everywhere}, \;\; \text{in}\;\; \mathbb{R}^2, \;\; \text{as}\;\; n \to \infty.
\end{equation*}
There also exists $h(x)\in H^1(\mathbb{R}^2) $ such that (see \cite[Proposition 2.7]{dMS})
\begin{equation*}
	|u_n(x)|\le h(x), \;\; \text{for almost all}\;\; x \in \mathbb{R}^2, \;\; \text{for all}\;\; n \in \mathbb{N}.
\end{equation*}
By Lemma $\ref{lem 4.3}$, for every $\varepsilon>0$, for every $q\ge1$, there is $C>0$ such that
\begin{equation*}
	\big|F(x,u_n(x))\big| \le \varepsilon \Big(h(x)\Big)^2 + C|h(x)|^q\big(e^{\beta (h(x))^2}-1\big), \;\; \text{almost everywhere}, \;\; \text{in}\;\; \mathbb{R}^2.
\end{equation*}
Since $h(x) \in H^1(\mathbb{R}^2)$, we have that $h(x)\in L^p(\mathbb{R}^2)$ for $2\le p<\infty$ and thus \[\Big(h(x)\Big)^2\in L^1(\mathbb{R}^2).\] On the other hand, by Remark $\ref{rm 5}$,
\begin{equation*}
	C\int_{\mathbb{R}^2} |h(x)|^q\Bigg( e^{\beta (h(x))^2} - 1\Bigg)dx < \infty.
\end{equation*}
Since $F(x,.)$ is continuous for every $x\in \mathbb{R}^2$, then
\begin{equation*}
	F(x, u_n(x)) \to F(x, u(x)),\;\; \text{almost everywhere},\;\; \text{in}\;\; \mathbb{R}^2.
\end{equation*}
The Lebesgue dominated convergence theorem implies that
\begin{equation}
	\label{eq 10}
	\int_{\mathbb{R}^2} F(x, u_n(x))\,dx \to \int_{\mathbb{R}^2}  F(x, u(x))\,dx, \;\; \text{as}\;\; n \to \infty.
\end{equation}
Therefore, \[J(u_n)\to J(u), \;\; \text{as}\;\; n \to \infty.\]
(4) \textbf{The functional $J$ is $O(2)-$invariant} on $E$ because, under assumption $(f_3)$, we have 
\begin{equation*}
    J(g_0u(x))= J(u(g_0^{-1}x))=J(u(x)).
\end{equation*}
Indeed, let $g_0\in O(2)$. Since the action of $O(2)$ on $E$ is isometric, then $\|g_0u\|=\|u\|$. Moreover, if $y=g_0^{-1}x$, then under assumption $(f_3)$, we have 
\begin{equation*}
    F(x, u(g_0^{-1}x)) =F(g_0y,u(y))
    = \int_{0}^{u(y)}f(g_0y,t)\,dt
    = \int_{0}^{u(y)}f(y,t)\,dt
    =F(y,u(y)).
\end{equation*}
	\end{proof}
	\begin{lemma}
		\label{lem 4.6}
Let $J$ be given by $(\ref{eq 5})$. Under assumptions $(f_0)$, $(f_1)$, for every $u\in H^1(\mathbb{R}^2)$ and for every $w \in C_{0}^{\infty}(\mathbb{R}^2)$ $($the set of infinitely differentiable functions with compact support in $\mathbb{R}^2$$)$, the Gateaux derivative of $J$ defined by
	\begin{equation*}
		J'(u)w:= \lim\limits_{t\to0} \dfrac{J(u+tw)-J(u)}{t} =\int_{\mathbb{R}^2} \Bigg(\nabla u \nabla w +uw  -f(x,u)w\Bigg)dx
	\end{equation*}
	exists and is continuous.
	\end{lemma}
	\begin{proof}
		\textbf{The functional $J$ is Gateaux differentiable}.	Let $u \in H^1(\mathbb{R}^2)$ and let $w \in C_{0}^{\infty}(\mathbb{R}^2)$. For $x\in \mathbb{R}^2$ and $|t|\in (0,1)$ we have 
		\begin{equation*}
			\frac{J(u+tw)-J(u)}{t} = (u,w)+\dfrac{t}{2}\|w\|^2 -\frac{1}{t} \Bigg(\int_{\mathbb{R}^2} \Big( F(x,u+tw)-F(x,u)\Big) \Bigg) dx.
		\end{equation*}
	By the mean value theorem, there exists $\alpha \in (0,1)$ such that 
		\begin{equation*}
\Big|	\frac{F(x,u+tw)-F(x,u)}{t} \Big| = \Big| f(x, u +\alpha t w) w \Big|.
		\end{equation*}
      By Lemma $\ref{lem 4.2}$, we have 
      \begin{equation*}
         \Big| f(x, u +\alpha t w) w \Big| \le 2\varepsilon |u+\alpha t w|\;|w| + c |w| \big(e^{\beta (u+\alpha tw)^2}-1\big). 
      \end{equation*}
      By Remark $\ref{rm 5}$, we see that $f(x, v)v\in L^1(\mathbb{R}^2)$, for every $v\in E$.
		Hence, the Lebesgue dominated convergence theorem and the continuity of $f$ (assumption $(f_0)$) imply 
		\begin{equation*}
				\lim\limits_{t\to 0} \frac{1}{t} \int_{\mathbb{R}^2} \Bigg(F(x,u+tw)-F(x,u)\Bigg)dx = \int_{\mathbb{R}^2} \lim\limits_{t \to 0} \Bigg( f (x, u +\alpha t w) w \Bigg)dx \\
				= \int_{\mathbb{R}^2} \Bigg(f(x, u) w \Bigg) dx.
		\end{equation*}
			It follows that  the functional $J$ is Gateaux differentiable and for all $w\in C_{0}^{\infty}(\mathbb{R}^2)$,
		\begin{equation*}
			J'(u)w= \int_{\mathbb{R}^2} \Bigg(\nabla u \nabla w +uw  -f(x,u)w\Bigg)dx.
		\end{equation*}
		\textbf{The Gateaux derivative $J'$ is continuous}. Assume that $(u_n)\to u$ in $H^1(\mathbb{R}^2)$. Then, by the Sobolev embedding theorem, $(u_n)\to u$ in $L^2(\mathbb{R}^2)$. 
		We have 
		\begin{align*}
			\Big|(J'(u_n)-J'(u))w\Big| &= \Big|\int_{\mathbb{R}^2} \Big( \Big(\nabla u_n -\nabla u\Big)w +(u_n-u)w   - \Big(f(x,u_n)-f(x,u)\Big)w\Big)\Big|\,dx\\
			&\le\|u_n-u\|\|w\| + \int_{\mathbb{R}^2} \Big| \Big(f(x,u_n)-f(x,u)\Big)w\Big|\,dx.
		\end{align*}
		Using the same arguments that led us to $(\ref{eq 10})$, we obtain that for every $w \in C_{0}^{\infty}(\mathbb{R}^2) $
		\begin{equation*}
			\int_{\mathbb{R}^2} f(x, u_n(x))w \,dx \to \int_{\mathbb{R}^2}  f(x, u(x))w\,dx, \;\; \text{as}\;\; n \to \infty.
		\end{equation*}
		It follows that 
		\[J'(u_n)-J'(u)\to 0, \;\; \text{as}\;\; n\to \infty.\]
	\end{proof}
	\begin{remark}
	    By Definition $\ref{def 4.1}$ and Lemma $\ref{lem 4.6}$, we see that $\mathbb{Z}_2-$critical points of the functional $J$ are exactly the weak solutions to problem $(P)$.
	\end{remark}
	\begin{lemma}
		\label{lem 4.7}
Under assumptions of Lemma $\ref{lem 4.5}$, for every $u\in H^1(\mathbb{R}^2)$, the equivariant weak slope $|d_{\mathbb{Z}_2}J|(u)$ satisfies
	\begin{equation*}
		|d_{\mathbb{Z}_2}J|(u)\ge \sup \Big\{  \int_{\mathbb{R}^2} \Big(\nabla u \nabla w +uw  -f(x,u)w\Big)\,dx \;|\; w \in C_{0}^{\infty}(\mathbb{R}^2), \; \|w\|\le 1      \Big \}.
	\end{equation*}
	\end{lemma}
	\begin{proof}
By Lemma $\ref{lem 4.5}$, the functional $J$ is well defined and continuous. Moreover, Lemma $\ref{lem 4.6}$ implies that for every $u\in H^1(\mathbb{R}^2)$ and $w \in C_{0}^{\infty}(\mathbb{R}^2)$ there exists $J'(u)$ such that
	\begin{equation*}
			J'(u)w= \int_{\mathbb{R}^2} \Bigg(\nabla u \nabla w +uw  -f(x,u)w\Bigg)dx
	\end{equation*}
	and the function $u\mapsto J'(u) w$ is continuous from $H^1(\mathbb{R}^2)$ into $\mathbb{R}$. From Remark $\ref{rmk 2}$ and Theorem $\ref{theo 2.1}$, we deduce that 
	\begin{equation*}
	-J'(u)w\le |d_{\mathbb{Z}_2}J|(u)\;\; \text{whenever}\;\; \|w\|\le 1. 
	\end{equation*}
	Since the function $w\mapsto J'(u) w$ is linear, then if we set $\zeta := -w$, we have 
	\begin{equation*}
		J'(u)\zeta = J'(u)(-w)=-J'(u)w\le |d_{\mathbb{Z}_2}J|(u)\;\; \text{whenever}\;\; \|\zeta\| =\|-w\|=\|w\|\le 1. 
	\end{equation*}
	\end{proof}
	
	\begin{lemma}
    \label{lem 4.8}
Under assumptions $(f_1), (f_2),(f_3)$ and $(f_4)$, the functional $J$ given in $(\ref{eq 5})$ satisfies $\mathbb{Z}_2-(PS)_c$ for every $c\in \mathbb{R}$.
	\end{lemma}
	\begin{proof}
	Let $(u_n)\subset H^1(\mathbb{R}^2)$ such that $d:=\sup\; J(u_n)< \infty$ and $|d_{\mathbb{Z}_2}J|(u_n)\to 0$, as $n\to \infty$. We want to show that the sequence $(u_n)$ has a convergent subsequence. By Lemma $\ref{lem 4.7}$,
	\begin{equation*}
	\|J'(u_n)\| \le |d_{\mathbb{Z}_2}J|(u_n) \;\; \text{with}\;\;	J'(u_n)w= \int_{\mathbb{R}^2} \Bigg(\nabla u_n \nabla w +u_nw  -f(x,u_n)w\Bigg)dx.
	\end{equation*}
	For sufficiently large $n\in \mathbb{N}$, we have $\|J'(u_n) \|\le1$ and $J(u_n)\le d+1$. Hence, for such $n$, we have 
	\begin{equation*}
		d+1 + \|u_n\| \ge J(u_n)- \dfrac{1}{\mu}J'(u_n)u_n,
        \end{equation*}
for every $\mu >2$.\\
    We have 
    \begin{align*}
    J(u_n)- \dfrac{1}{\mu}J'(u_n)u_n &= \dfrac{1}{2}\|u_n\|^2 -\int_{\mathbb{R}^2} F(x,u_n)dx -\dfrac{1}{\mu} \int_{\mathbb{R}^2} \Bigg(\nabla u_n \nabla u_n +|u_n|^2  -f(x,u_n)u_n\Bigg)dx\\
		&= \Big( \dfrac{1}{2}- \dfrac{1}{\mu} \Big)\|u_n\|^2 + 	\int_{\mathbb{R}^2}\Big( \frac{1}{\mu}f(x, u_n)u_n -F(x,u_n)\Big)  dx\\
		&\ge\Big( \dfrac{1}{2}- \dfrac{1}{\mu} \Big)\|u_n\|^2 \;\; (\text{by assumption}\;\; (f_2)).
	\end{align*}
    
    It follows that 
    \begin{equation*}
       d+1 + \|u_n\| \ge  \Big( \dfrac{1}{2}- \dfrac{1}{\mu} \Big)\|u_n\|^2.
    \end{equation*}
    
	Thus $(u_n)$ is bounded in $H^1(\mathbb{R}^2)$. Since $(u_n)$ is bounded, there exists a subsequence still denoted $(u_n)$, and there is $u$ in $H^1(\mathbb{R}^2)$ such that 
    \begin{equation*}
	(u_n) \to u \;\; \text{weakly in}\;\; E, \;\; \text{as}\;\; n \to \infty, 
\end{equation*}
\begin{equation*}
	(u_n) \to u, \;\; \text{in}\;\; L_{loc}^p(\mathbb{R}^2), \;\;\text{for}\;\; 2\le p<\infty,\;\; \text{as}\;\; n \to \infty,
\end{equation*}
\begin{equation*}
	(u_n(x)) \to u(x), \;\; \text{almost everywhere}, \;\;\text{in}\;\; \mathbb{R}^2, \;\; \text{as}\;\; n \to \infty.
\end{equation*}

    We have 
    \begin{multline*}
        J'(u_n)(u_n-u) -J'(u)(u_n-u)= \int_{\mathbb{R}^2} \Bigg(\nabla u_n \nabla (u_n-u) +u_n(u_n-u) -f(x,u_n)(u_n-u)\Bigg)dx \\- \int_{\mathbb{R}^2} \Bigg(\nabla u \nabla (u_n-u)+u(u_n-u)  -f(x,u)(u_n-u)\Bigg)dx.
    \end{multline*}
    
    This implies that
    \begin{multline*}
     J'(u_n)(u_n-u) -J'(u)(u_n-u) = \int_{\mathbb{R}^2} \Bigg( \nabla(u_n-u) \nabla(u_n-u) +|u_n-u|^2\Bigg)dx \\  -\int_{\mathbb{R}^2} \Big(f(x,u_n)-f(x,u)\Big)(u_n-u)\,dx.
    \end{multline*}
    
    That is, 
    \begin{equation*}
         J'(u_n)(u_n-u) -J'(u)(u_n-u) = \|u_n-u\|^2 - \int_{\mathbb{R}^2} \Big(f(x,u_n)-f(x,u)\Big)(u_n-u)\,dx.
    \end{equation*}
    
    Hence,
	\begin{equation*}
		\|u_n-u\|^2 =   J'(u_n)(u_n-u) -J'(u)(u_n-u)+ \int_{\mathbb{R}^2} \Big(f(x,u_n)-f(x,u)\Big)(u_n-u)\,dx.
	\end{equation*}
    
On the one hand, we claim that 
	\begin{equation*}
	\int_{\mathbb{R}^2} f(x, u_n(x))u_n(x) \,dx \to \int_{\mathbb{R}^2}  f(x, u(x))u(x)\, dx, \;\; \text{as}\;\; n \to \infty.
\end{equation*}
Indeed:\\
$(i)$ By Lemma $\ref{lem 4.2}$, for every $\varepsilon>0$, there is $c>0$ such that 
\begin{equation*}
	\big|f(x,u_n(x)) u_n(x)\big| \le 2\varepsilon |u_n(x)|^2 + c|u_n(x)|\big(e^{\beta (u_n(x))^2}-1\big), \;\; \text{almost everywhere},\;\;\text{in}\;\; \mathbb{R}^2.
\end{equation*}
$(ii)$ Since $f(x,.)$ is continuous for every $x\in \mathbb{R}^2$, then
\begin{equation*}
	f(x, u_n(x))u_n(x) \to f(x, u(x))u(x),\;\;\text{almost everywhere},\;\; \text{in}\;\; \mathbb{R}^2.
\end{equation*}
and
\[ 2\varepsilon |u_n(x)|^2 + c|u_n(x)|\big(e^{\beta (u_n(x))^2}-1\big) \to 2\varepsilon |u(x)|^2 + c|u(x)|\big(e^{\beta (u(x))^2}-1\big) , \;\; \text{almost everywhere},\;\;\text{in}\;\; \mathbb{R}^2.\]
$(iii)$ By Remark $\ref{rm 5}$, 
\[\lim_{n\to \infty} \int_{\mathbb{R}^2} \Big(2\varepsilon |u_n(x)|^2 + c|u_n(x)|\big(e^{\beta (u_n(x))^2}-1\big) \Big)dx = \int_{\mathbb{R}^2}\Big( 2\varepsilon |u(x)|^2 + c|u(x)|\big(e^{\beta (u(x))^2}-1\big)\Big)dx <\infty.\]
 By Generalized Lebesgue dominated convergence theorem, we
obtain
\begin{equation*}
	\int_{\mathbb{R}^2} f(x, u_n(x)) u_n(x)\, dx \to \int_{\mathbb{R}^2}  f(x, u(x))u(x)\, dx, \;\; \text{as}\;\;n \to \infty,
\end{equation*} It follows that
\begin{equation*}
	\int_{\mathbb{R}^2} \Big(f(x,u_n)-f(x,u)\Big)(u_n-u)\,dx \to 0, \;\; \text{as}\;\; n\to \infty.
\end{equation*}

On the other hand, since  $(u_n)\rightharpoonup u$, as $n\to \infty$, then \begin{equation*}
	 J'(u_n)(u_n-u) -J'(u)(u_n-u) \to 0, \;\; \text{as}\;\; n\to \infty.
\end{equation*}

We conclude that 
\begin{equation*}
	\|u_n-u\|^2 \to 0, \;\; \text{as}, \;\; n\to \infty.
\end{equation*}
Therefore, $(u_n)$ has a convergent subsequence and
by Definition $\ref{def 2.6}$, the functional $J$ satisfies $\mathbb{Z}_2-(PS)_c$ for every $c\in \mathbb{R}$.
	\end{proof}
	\begin{lemma}[see \cite{Wi}]
		\label{lem 4.9}
	If $2 < p< \infty$, then for every $u\in E$, we have 
	\begin{equation*}
		\alpha_k(p):= \underset{\underset{\|u\|=1}{u\in Z_k}}{\sup}\; |u|_p \to 0, \; k\to \infty.
	\end{equation*}
	\end{lemma}
\subsection{Proof of Theorem $\ref{theo 4.3}$} Now we are ready to prove Theorem $\ref{theo 4.3}$.
\begin{proof}
From $(\star)$, assumption $(A_1)$ of Theorem $\ref{theo 3.2}$ is satisfied. By Lemma $\ref{lem 4.8}$, the geometric assumption $(A_4)$ also holds for $f=J$ and $G=\mathbb{Z}_2$. To prove Theorem $\ref{theo 4.3}$, it now suffices to verify the geometric assumptions $(A_2)$ and $(A_3)$ for $f=J$, which then allows us to apply Theorem $\ref{theo 3.2}$ and obtain $\mathbb{Z}_2-$critical points of the functional $J$ restricted to $E$.\\

		From $(f_2)$ (see \cite{ZL}), there is $C_4>0$ such that 
		\begin{equation}
			\label{eq 11}
			F(x,t)\ge C_4 |t|^{\gamma}.
		\end{equation}

Let $u\in Y_k$. In view of $(\ref{eq 11})$, we have 
		\begin{align*}
			J(u)&= \dfrac{\|u\|^2}{2} - \int_{\mathbb{R}^2} F(x,u) dx\\
			&\le\dfrac{\|u\|^2}{2} -C_4|u|_{\gamma}^\gamma.
		\end{align*}
		Since on the finite-dimensional spaces $Y_k$ all norms are equivalent (there is $C_5>0$ such that $ \|u\| \le C_5|u|_{\gamma}$), we obtain 
		\begin{equation*}
			J(u)\le \dfrac{\|u\|^2}{2} -C_6\|u\|^\gamma,
		\end{equation*}
		for some $C_6>0$.
		Thus, the relation $(A_2)$ in Theorem $\ref{theo 3.2}$ is satisfied for $f=J$ for every $\rho_k>0$ large enough. \\
		From Lemma $\ref{lem 4.3}$, for any $\varepsilon>0$, for each $\delta\ge 1$, there exists $C>0$ such that
		\begin{equation}
        \label{eq 12}
			|F(x,u)|\le \varepsilon \dfrac{u^2}{4} + C|u|^\delta\big(e^{\beta u^2}-1\big).
		\end{equation}
		
	From  Remark $\ref{rm 5}$, we have \begin{equation}
    \label{eq 13}
	\int_{\mathbb{R}^2} C|u|^\delta\big(e^{\beta u^2}-1\big)dx \le c_0|u|_{\tau}^{\delta},
	\end{equation}
  where  $ 2\le \tau< \infty$ and for some $c_0>0$.

    On $Z_k$, for $2<p<\infty$ and for $2< \tau <\infty$, we have 
	\begin{align*}
J(u) &\ge \dfrac{\|u\|^2}{2} -\int_{\mathbb{R}^2} |F(x,u)|dx\\
&\ge\dfrac{\|u\|^2}{2} - \dfrac{\varepsilon}{4} \int_{\mathbb{R}^2} |u|^2 dx - c_0|u|^p_{\tau}\;\; \text{(by equations (\ref{eq 12}) and (\ref{eq 13})})\\
&= \dfrac{\|u\|^2}{2} -\dfrac{\varepsilon}{4}|u|_{2}^2- c_0 |u|^p_{\tau}.
	\end{align*}
	By the Sobolev embedding theorem, there is $c_2>0$ such that $|u|_2^2 \le c_2 \|u\|^2$. We may choose $\varepsilon$ so that $c_2 \varepsilon<1$. Then,
    \begin{equation*}
        J(u) \ge \dfrac{\|u\|^2}{2} -\dfrac{\|u\|^2}{4} -c_0 |u|^p_{\tau}
    \end{equation*}
For $\alpha_k(p)$ from Lemma $\ref{lem 4.9}$, let us choose $r_k:=\Bigg( 2c_0 p \Big(\alpha_k(p)\Big) \Bigg)^{\frac{1}{2-p}}$. For $v \in Z_k$ with $\|v\|=r_k$ (in fact, we take all $v:= r_k u$ with $\|u\|=1$), we obtain
    \begin{align*}
        J(v) &\ge \dfrac{\|r_ku\|^2}{4} -c_0|r_ku|^p_\tau \\
        &= \dfrac{(r_k)^{2}}{4}\|u\|^2-c_0(r_k)^p|u|^p_{\tau}.
    \end{align*}
    Since $\|u\|=1$, we have
    \begin{equation*}
        |u|^p_{\tau} \le \Big(\alpha_k(p)\Big)^p.
    \end{equation*}
    
    Since $\frac{(r_k)^{2-p}}{p}= 2c_0\Big(\alpha_k(p)\Big)^p$, then 
\begin{align*}
J(v) &\ge \dfrac {(r_k)^{2}}{4} - c_0(r_k)^p\Big(\alpha_k(p)\Big)^p\\
&=\Big( \dfrac{1}{4}-\dfrac{1}{2p}\Big)(r_k)^2\\
	&=\Big(\dfrac{1}{4}-\dfrac{1}{2p}\Big)\Bigg(2c_0p\Big(\alpha_k(p)\Big)^p\Bigg)^\frac{2}{2-p}.
\end{align*}
	Since by Lemma $\ref{lem 4.9}$, $\alpha_k \to 0$, as $k\to \infty$, and $\frac{1}{4}-\frac{1}{2p}>0$, we obtain that $J(v) \to  +\infty$ if $v\in Z_k$ such that $\|v\| =r_k$. The relation $(A_3)$ in Theorem $\ref{theo 3.2}$ is thus proved for $f=J$.\\
    
	By Lemma $\ref{lem 4.5}\, (2)$, the functional $J$ is $\mathbb{Z}_2-$invariant. Using Theorem $\ref{theo 3.2}$, the functional $J{\Big|_E}$ admits an unbounded sequence of $\mathbb{Z}_2-$critical values. In orther words, problem $(P)$ possesses a sequence of radial solutions $(u_k)$ such that $J(u_k)\to \infty$, as $k\to \infty$. By the principle of symmetric criticality (Theorem $\ref{theo 4.2}$), $\mathbb{Z}_2-$critical points of $J$ in $E$ are $\mathbb{Z}_2-$critical points of $J$ in $H^1(\mathbb{R}^2)$. Hence, problem $(P)$ has a sequence of solutions $(u_k)$ in $H^1(\mathbb{R}^2)$ such that $J(u_k)\to \infty$, as $k\to \infty$.  This completes the proof of Theorem $\ref{theo 4.3}$.
    \end{proof}
    
	\section*{Conclusion}
In this work, we presented a continuous version of the classical Fountain theorem. Using the deformation lemma for continuous functionals established in \cite{KV}, we proved a minimax prinple theorem, which played a key role in the proof of the aforementioned result. In future work, we aim to generalize our result to strongly indefinite continuous functionals, that is, continuous functionals of the form $f(u) = \dfrac{1}{2} \langle Lu, u \rangle - \Psi(u)$ defined on a Hilbert
 space $X$, where $L : X \to X$ is a self-adjoint operator whose negative and positive eigenspaces are both infinite-dimensional. The study of such functionals is motivated by several problems in mathematical physics.
\section*{Data Availability Statement}
Data sharing is not applicable to this article as no datasets were generated or analyzed.
\section*{Conflicts of Interest}
The authors declare no conflicts of interest.
\section*{Author Contributions}
All authors made essential contributions to the research.
\section*{Funding}
 This work was funded by a  grant from the Natural Sciences and Engineering Research Council of Canada.
\section*{Acknowledgments}
\textit{Software Used}. No AI has been used to prepare the manuscript, as a whole or in part.

\end{document}